\documentclass[a4paper,11pt]{article}
\usepackage[english, french]{babel}
\usepackage[T1]{fontenc}
\usepackage{fancyhdr} 

\usepackage{amsmath,amssymb}
\usepackage{latexsym,enumerate}
\usepackage{amsmath,amsthm,amsopn,amstext,amscd,amsfonts,amssymb}
\usepackage{graphicx,natbib}

\newtheorem{theorem}{Theorem}

\newtheorem{lemma}{Lemma}

\newtheorem{proposition}{Proposition}

\renewcommand{\r}{\mathbb R}
\renewcommand{\l}{\mathbb L}

\newcommand{\Tr}{\text{tr}}
\newcommand{\R}{\mathbb{R}}
\newcommand{\abs}[1]{{\left\lvert#1\right\rvert}} %
\newcommand{\deq}{\mathrel{\mathop:} = } 
\newcommand{\eqd}{= \mathrel{\mathop:} } 
\newcommand{\paren}[1]{{\left( #1 \right)}} %
\newcommand{\norm}[2]{{\left\lVert #2 \right\rVert}_{#1}}
\newcommand{\Fourier}[1]{\widetilde{#1}}
\newcommand{\N}{\mathbb{N}}
\def \1{{\rm 1}\mskip -4,5mu{\rm l} }

\title{State estimation in quantum homodyne tomography with noisy data}
\author{J M Aubry$^1$, C Butucea$^2$ and K Meziani$^3$\\
$^1$ \footnotesize{Laboratoire d'Analyse et de Mathématiques Appliquées (UMR CNRS 8050),}\\
\footnotesize{Université Paris-Est, 94010 Créteil Cedex, France}\\
$^2$ \footnotesize{Laboratoire Paul Painlevé (UMR CNRS 8524),}\\
\footnotesize{Université des Sciences et Technologies de Lille 1, 59655 Villeneuve d'Ascq Cedex, France}\\
$^3$ \footnotesize{Laboratoire de Probabilités et Modèles Aléatoires,}\\
\footnotesize{Université Paris VII (Denis Diderot), 75251 Paris Cedex 05, France}\\
\footnotesize{\textbf{email:} jmaubry@math.cnrs.fr; cristina.butucea@math.univ-lille1.fr; meziani@math.jussieu.fr}}  
\begin{document}
\maketitle


\begin{abstract}
    In the framework of noisy quantum homodyne tomography with
    efficiency parameter $0 < \eta \leq 1$, we propose two estimators
    of a quantum state whose density matrix elements $\rho_{m,n}$
    decrease like $e^{-B(m+n)^{ r/ 2}}$, for fixed known $B>0$ and
    $0<r\leq 2$.  The first procedure estimates the matrix
    coefficients by a projection method on the pattern functions (that
    we introduce here for $0<\eta \leq 1/2$), the second procedure is
    a kernel estimator of the associated Wigner function.  We compute
    the convergence rates of these estimators, in $\mathbb{L}_2$ risk.
\end{abstract}

\section{Introduction}

\noindent Experiments in quantum optics consist in creating, manipulating and
measuring quantum states of light. The technique called quantum
homodyne tomography allows to retrieve partial, noisy information from
which the state is to be recovered: this is the subject of the present
chapter.

\subsection{Quantum states}

\noindent Mathematically, the main concepts of quantum mechanics are formulated
in the language of selfadjoint operators acting on Hilbert spaces.  To
every quantum system one can associate a complex Hilbert space
$\mathcal{H}$ whose vectors represent the wave functions of the
system. These vectors are identified to projection operators, or pure
states.  In general, a state is a mixture of pure states described by
a compact operator $\rho$ on $\mathcal{H}$ having the following
properties:
\begin{enumerate}
  \item Selfadjoint: $\rho=\rho^{*}$, where $\rho^{*}$ is the adjoint
    of $\rho$.
  \item Positive: $\rho\geq 0$, or equivalently $\langle\psi,
    \rho\psi\rangle\geq 0$ for all $\psi\in\mathcal{H}$.
  \item Trace one: $\Tr(\rho)=1$.
\end{enumerate}

\noindent When $\mathcal{H}$ is separable, endowed with a countable orthonormal
basis, the operator $\rho$ is identified to a \emph{density matrix}
$[\rho_{m,n}]_{m,n\in\N}$.

\noindent The positivity property implies that all the eigenvalues of $\rho$ are
nonegative and by the trace property, they sum up to one.  In the case
of the finite dimensional Hilbert space $\mathbb{C}^{d}$, the density
matrix is simply a positive semi-definite $d \times d$ matrix of trace
one.  Our setup from now on will be $\mathcal{H} = L^{2}(\R)$, in
which case we employ the orthonormal Fock basis made of the Hermite
functions
\begin{equation}
    \label{eq:2}
    h_m(x) \deq (2^m m! \sqrt\pi)^{-\frac12} H_m(x) e^{-\frac{x^2}{2}}
\end{equation}
where $H_m(x) \deq (-1)^m e^{x^2} \frac{d^m}{dx^m} e^{-x^2}$ is the
$m$-th Hermite polynomial.  Generalizations to higher dimensions are
straightforward.

\noindent To each state $\rho$ corresponds a \emph{Wigner distribution}
$W_{\rho}$, which is defined via its Fourier transform in the way
indicated by equation (\ref{def.Wigner}):
\begin{equation}
    \label{def.Wigner}
    \widetilde{W}_\rho(u,v) \deq \iint e^{-i(uq+vp)}W_\rho(q,p) dq dp
    \deq
    \mathrm{Tr}\big(\rho \exp(-iu\mathbf{Q}-iv\mathbf{P})\big)
\end{equation}
where $\mathbf{Q}$ and $\mathbf{P}$ are canonically conjugate
observables (e.g. electric and magnetic fields) satisfying the
commutation relation $[\mathbf{Q}, \mathbf{P}] = i$ (we assume a
choice of units such that $\hbar = 1$).  It is easily checked that
$W_{\rho}$ is real-valued, has integral $\iint_{\R^2} W_\rho(q,p)dqdp=1$ and uniform bound $\abs{W_{\rho}(q,p)} \leq
\frac{1}{\pi}$.

\noindent For any $\phi \in \R$, the Wigner distribution allows one to
easily recover the probability density $x \mapsto p_{\rho}(x,\phi)$ of
$\mathbf{Q} \cos \phi + \mathbf{P} \sin \phi$ by
\begin{equation}
    \label{eq:prhophi}
    p_{\rho}(x,\phi) =\mathcal{R}[W_{\rho}] (x,\phi),
\end{equation}
where $\mathcal{R}$ is the Radon transform defined in equation
(\ref{eq.Radon.transform})
\begin{equation}\label{eq.Radon.transform}
    \mathcal{R}[W_{\rho}](x,\phi)=\int_{-\infty}^\infty W_{\rho}(
    x\cos\phi - t\sin\phi, \, x \sin\phi + t\cos\phi )dt.
\end{equation}

\noindent Moreover, the correspondence between $\rho$ and $W_{\rho}$ is one to
one and isometric with respect to the $\mathbb{L}_{2}$ norms as in
equation (\ref{eq.l2.isometry}):
\begin{equation}
    \label{eq.l2.isometry}
    \| W_{\rho}\|_2^2 \deq \iint |W_{\rho}(q,p)
    |^2 dq dp = \frac{1}{2\pi}\| \rho \|_{2}^{2} \deq
    \frac{1}{2\pi}\sum_{j, k=0}^{\infty} |\rho_{jk}|^2.
\end{equation}

\noindent From now on we denote by $\langle \cdot, \cdot \rangle$ and $\| \cdot
\|$ the usual Euclidian scalar product and norm, while $C(\cdot)$ will
denote positive constants depending on parameters given in the
parentheses.

\noindent We suppose that the unknown state belongs to the class
$\mathcal{R}(B,r)$ for $B>0$ and $0 < r \leq 2$ defined by
\begin{equation}\label{eq.classcoeff}
    \mathcal{R}(B,r) \deq \{\rho {\rm \ quantum \ state} :
    |\rho_{m,n} |\leq \exp(-B (m+n)^{r/2})\}.
\end{equation}

\noindent For simplicity, we have chosen to express the results relative to a
class which is the intersection of the (positive) ball of radius 1 in
some Banach space with the hyperplane $\Tr(\rho) = 1$. Another radius
for the class would only change the constant $C$ in front of the
asymptotic rates of convergence that we will find.

\noindent As it will be made precise in Propositions~\ref{prop:popetown}
and~\ref{prop:popetown2}, quantum states in the class given in
(\ref{eq.classcoeff}) have fast decreasing and very smooth Wigner
functions.  From the physical point of view, the choice of such a class
of Wigner functions seems to be quite reasonable considering that
typical states $\rho$ prepared in the laboratory do satisfy this type
of condition. 

\subsection{Statistical model}

\noindent Let us describe the statistical model.  Consider $(X_{1}, \Phi_{1}),
\dots ,(X_{n}, \Phi_{n})$ independent identically distributed random
variables with values in $\mathbb{R}\times [0,\pi]$ and distribution
$P_{\rho}$ having density $p_{\rho}(x,\phi)$ (given by (\ref{eq:prhophi})
 with respect to $\frac{1}{\pi}\lambda$, $\lambda$
being the Lebesgue measure on $\mathbb{R}\times [0,\pi]$.  The aim is
to recover the density matrix $\rho$ and the Wigner function $W_\rho $
from the observations.

\noindent However, there is a slight complication.  What we observe are not the
variables $(X_{\ell}, \Phi_{\ell})$ but the noisy ones $(Y_{\ell},
\Phi_{\ell})$, where
\begin{equation}
    \label{noisy.data}
    Y_\ell:=\sqrt{\eta}X_\ell+\sqrt{(1-\eta)/2}~\xi_\ell,
\end{equation}
with $\xi_\ell$ a sequence of independent identically distributed
standard Gaussians which are independent of all $(X_j, \Phi_{j})$.
The detection efficiency parameter $0 < \eta \leq 1$ is known from the
calibration of the apparatus and we denote by $N^\eta$ the centered
Gaussian density of variance $(1-\eta)/2$, and $\widetilde N^\eta$ its
Fourier transform.  Then the density $p_{\rho}^{\eta}$ of $(Y_{\ell},
\Phi_{\ell})$ is given by the convolution of the density
$p_\rho(\cdot/\sqrt{\eta}, \phi)/\sqrt{\eta}$ with $N^\eta$
\begin{eqnarray*}
    p^\eta_\rho(y,\phi) &=& \int_{-\infty}^\infty
    \frac{1}{\sqrt{\eta}} p_\rho\left(\frac{y-x}{\sqrt{\eta}}, \phi
    \right) N^\eta (x)dx\\
    &~ \eqd& \left(\frac{1}{\sqrt\eta}
        p_\rho\left(\frac{\cdot}{\sqrt\eta},\phi\right) \ast N^\eta
    \right)(y).
\end{eqnarray*}
In the Fourier domain this relation becomes
\begin{eqnarray}
    \label{fourierproun}
    \mathcal{F}_1[p^\eta_\rho(\cdot,\phi)](t)
    &=& \mathcal{F}_1[p_\rho(\cdot,\phi)](t \sqrt{\eta})
    \widetilde{N}^\eta(t),
\end{eqnarray}
where $\mathcal{F}_{1} $ denotes the Fourier transform with respect to
the first variable.

\noindent The theoretical foundation of quantum homodyne tomography was outlined
in \cite{Vogel&Risken} and has inspired the first experiments
determining the quantum state of a light field, initially with optical
pulses in \cite{Smithey,Smitheybis,Leonhardt}. The reconstruction of
the density from averages of data has been discussed or studied in
\cite{DAriano.0,DAriano.2,DAriano.3,Artiles&Gill&Guta} for $\eta=1$
(no photon loss).  Max-likelihood methods have been studied in
\cite{BDPS,Artiles&Gill&Guta,DMS,Guta} and procedure using adaptive
tomographic kernels to minimize the variance has been proposed in
\cite{DP}. The estimation of the density matrix of a quantum state of
light in case of efficiency parameter $\frac12 < \eta \leq 1$ has been
discussed in \cite{DAriano.1,DMS,DAriano.5} and considered in
\cite{Richter} via the pattern functions for the diagonal elements.

\subsection{Outline of the results}

\noindent The goal of this chapter is to define estimators of both the density
matrix and the Wigner function and to compare their performance in
$\mathbb{L}_2$ risk. In order to compute estimation risks and to tune
the underlying parameters, we define a realistic class of quantum
states $\mathcal{R}(B,r)$, depending on parameters $B>0$ and $ 0<
r\leq 2$, in which the elements of the density matrix decrease
rapidly.

\noindent In Section \ref{sec.dec.reg}, we prove that the fast decay of the
elements of the density matrix implies both rapid decay of the Wigner
function and of its Fourier transform, allowing us to translate the
classes $\mathcal{R}(B,r)$ in terms of Wigner functions.

\noindent In Section \ref{sec.dens.mat}, we give estimators of the density
matrix $\rho$.  The legend was somehow forged that no estimation of
the matrix is possible when $0 < \eta \leq 1/2$. The physicists argue
that their machines actually have high detection efficiency, around
0.8; it is nevertheless satisfying to be able to solve this problem in
any noise condition.  We give here the so-called \emph{pattern functions} 
to use for estimating the density matrix in the noisy
case with \emph{any} value of $\eta$ between 0 and 1. These pattern
functions allow us to solve an inverse problem which becomes (severly)
ill-posed when $0 < \eta \leq 1/2$.  In this case, we regularize the
inverse problem and this introduces a smoothing parameter which we
will choose in an optimal way. We compute the upper bounds for the
rates achieved by our methods, with $\mathbb{L}_2$ risk measure.

\noindent In Section \ref{sec.Wigner}, we study a kernel estimator of the Wigner
function in $\mathbb{L}_2$ risk, over the same class of Wigner
functions. It is a truncated version of the estimator in
\cite{Butucea&Guta&Artiles} and tuned accordingly. We compute upper
bounds for the rates of convergence of this estimator in $\mathbb{L}_2$ risk.

\noindent To conclude, we may infer that the performances of both estimators are
comparable. We obtain nearly polynomial rates for the case $r = 2$ and
intermediate rates for $0 < r < 2$ (faster than any logarithm, but
slower than any polynomial).  It is convenient to have methods to
estimate directly both representations of a quantum state.  The
estimator of the matrix $\rho$ can be more easily projected on the
space of proper quantum states. On the other hand, we may capture some
features of the quantum states more easily on the Wigner function, for
instance when this function has significant negative parts, the fact
that the quantum state is non classical.

\section{Decrease and smoothness of the Wigner distribution}
\label{sec.dec.reg}

\noindent We recall that the Wigner distribution $W_{\rho}$ was defined in the
introduction.  In the Fock basis, we can write $W_{\rho}$ in terms of
the density matrix $[\rho_{m,n}]$ as follows (see Leonhardt
\cite{Leonhardt} for the details).
\begin{equation}
    W_{\rho}(q,p) = \sum_{m,n} \rho_{m,n} W_{m,n}(q,p) \nonumber
\end{equation}
where
\begin{equation}
    \label{eq:1}
    W_{m,n}(q,p) = \frac1\pi \int e^{2ipx} h_m(q-x) h_n(q+x) dx.
\end{equation}
It can be seen that $W_{m,n}(q,p) = W_{n,m}(q,-p)$ and
if $m \geq n$,
\begin{eqnarray}
   W_{m,n}(q,p)&=&\frac{(-1)^m}{\pi} \paren{\frac{n!}{m!}}^\frac12
    e^{-\paren{q^2+p^2}} \nonumber\\
    && \times \paren{\sqrt2(ip-q)}^{m-n}
    L_n^{m-n}\paren{2q^2+2p^2}
    \label{eq:WmnLaguerre}
\end{eqnarray}
thus, writing $z \deq \sqrt{q^2+p^2} $,
\begin{equation}
    \label{eq:deflmn}
    l_{m,n}(z) \deq \abs{W_{m,n}(q,p)} = \frac{2^{\frac{m-n}2}}{\pi}
    \paren{\frac{n!}{m!}}^\frac12 e^{-z^2} z^{m-n}
    \abs{L_n^{m-n}(2z^2)}
\end{equation}
where $L_{n}^{\alpha}(x) \deq (n!)^{-1} e^x x^{-\alpha} \frac{d^n}{dx^n}
(e^{-x} x^{n+\alpha})$ is the Laguerre polynomial of degree $n$ and
order $\alpha$.  Concerning the Fourier transforms, we also recall
that
\begin{equation}
    \Fourier{W_{m,n}}(q,p) = \frac{(-i)^{m+n}}{2}
    {W_{m,n}}\paren{\frac{q}{2},\frac{p}{2}}.
    \label{eq:FourierWmn}
\end{equation}

\noindent In this section we show how a decrease condition on the coefficients
of the density matrix translates on the corresponding Wigner
distribution.  First the case $r < 2$:
\begin{proposition}
    \label{prop:popetown}
    Assume that $0 < r < 2$ and that there exists $B > 0$ such that,
    for all $m \geq n$,
    \begin{equation}
        \abs{\rho_{m,n}} \leq e^{-B(m+n)^{r/2}}.  \nonumber
    \end{equation}
    Then for all $\beta < B$, there exists $z_{0}$
    (depending explicitly on $r, B, \beta$, see proof)
    such that $z := \sqrt{q^2 + p^2} \geq z_{0}$ implies
    \begin{equation}
        \label{eq:dec}
        \abs{W_{\rho}(q,p)} \leq A(z) e^{- \beta z^r}
    \end{equation}
    as well as
    \begin{equation}
        \label{eq:reg}
        \abs{\Fourier{W_{\rho}}(q,p)} \leq A(z/2) e^{- \beta (z/2)^r}
    \end{equation}
    where $A(z) \deq \frac1\pi \paren{\sum\limits_{m,n} e^{-B(m+n)^{r/2}} +
      \frac{4}{B r} z^{4-r}}$.
\end{proposition}

\noindent If $r=2$, the result is a little different:
\begin{proposition}
    \label{prop:popetown2}
    Suppose that there exists $B > 0$ such that, for all $m \geq n$,
    \begin{equation}
        \abs{\rho_{m,n}} \leq e^{-B(m+n)}.  \nonumber
    \end{equation}
    Then there exists $z_{0}$ such that $z := \sqrt{q^2 + p^2} \geq
    z_{0}$ implies
    \begin{equation}
        \label{eq:dec2}
        \abs{W_{\rho}(q,p)} \leq A(z) e^{- \frac{B}{(1+\sqrt{B})^2}
          z^2}
    \end{equation}
    as well as
    \begin{equation}
        \label{eq:reg2}
        \abs{\Fourier{W_{\rho}}(q,p)} \leq A(z/2) e^{-
          \frac{B}{(1+\sqrt{B})^2} (z/2)^2}
    \end{equation}
    for $A(z) = \frac1\pi \paren{\sum\limits_{m,n} e^{-B(m+n)} + \frac{2
        e^{B}}{B(1+\sqrt{B})^2} z^{2}}$.
\end{proposition}

\noindent Note that $\frac{B}{\paren{1+\sqrt{B}}^2} < \min(B,1)$.  Even when $B$
is very large, we cannot hope to obtain a faster decrease because
$e^{-z^2}$ is the decrease rate of the basis functions themselves
(Lemma \ref{lemm:estimeLaguerre}). 

\noindent The proof of these propositions is defered to
Appendix~\ref{sec.proofs}.  More general results and converses are
studied in~\cite{Aubry:2007fk}. Let us now state a few general utility
lemmata.

\begin{lemma}
    \label{lemm:ineqdiff}
    Let $y$ and $w$ be two $C^2$ functions: $[x_0,+\infty) \to
    (0,+\infty)$ such that $y'(x) \to 0$, $w$ is bounded, satisfying
    the differential equations
    \begin{eqnarray*}
        y''(x) &=& \phi(x) y(x) \\
        w''(x) &=&\psi(x) w(x),
    \end{eqnarray*}
    with continuous $\phi(x) \leq \psi(x)$, and initial conditions
    $y(x_0) = w(x_0)$.  Then for all $x \geq x_0$, $w(x) \leq y(x)$.
\end{lemma}

\begin{proof}
    Suppose that there exists $x_1 \geq x_0$ where $w(x_1) > y(x_1)$.
    Then for some $x_2 \in[x_0,x_1]$ we have $w'(x_2) > y'(x_2)$ and
    $w(x_2) \geq y(x_2)$.  Consequently, for all $x \geq x_2$, $w''(x)
    - y''(x) \geq 0$, and $w'(x)-y'(x) \geq w'(x_2) - y'(x_2)$.  When
    $x\to\infty$, $\liminf w'(x) \geq w'(x_2) - y'(x_2) > 0$, which
    contradicts the boundedness of $w$.
\end{proof}

\noindent This lemma is used to prove a bound on the Laguerre functions.

\begin{lemma}
    \label{lemm:estimeLaguerre}
    For all $m, n \in \N$ and $s \deq \sqrt{m+n+1}$, for all $z \geq 0$,
    \begin{equation}
     l_{m,n}(z) \leq \frac1\pi\left\{ \begin{array}{ll}
                                 1 &{\rm \ if \ } 0 \leq z \leq s \\
                                 e^{-(z-s)^2} &{\rm \ if \ } z \geq s.
                                       \end{array} 
                               \right. 
	    \label{eq:estimeLaguerre}
     \end{equation}
\end{lemma}

\begin{proof}
    When $z \leq s$, the result follows from the uniform bound on
    Wigner functions obtained by applying the Cauchy-Schwarz
    inequality to~(\ref{eq:1}).

    When $z \geq s$, $L_{n}^\alpha(2z^2)$ doesn't vanish and keeps the
    same sign as $L_{n}^\alpha(2s^2)$.  Now, as it can be seen
    from~\cite[5.1.2]{Szego:1959uq}, the function $w(z) \deq
    \sqrt{z}{l_{m,n}(z)}$ satisfies the differential equation
    $w''=(4(z^2-s^2)+\frac{\alpha^2-1/4}{z^2})z$.  On the other hand,
    $y(z) \deq \sqrt{s} {l_{m,n}(s)} e^{-(z-s)^2}$ satisfies $y'' =
    (4(z-s)^2-2) y$.  When $z \geq s$,
    \begin{equation}
        4(z-s)^2-2 < 4(z^2-s^2) + \frac{\alpha^2-1/4}{z^2}
    \end{equation}
    from which we conclude with Lemma~\ref{lemm:ineqdiff} that $w(z)
    \leq y(z)$.
\end{proof}

\noindent Finally, a lemma to bound the tail of a series.

\begin{lemma}
    \label{lemm:magicjohnson}
    If $\nu > 0$ and $C > 0$, there exists a $z_{0}$ such that $z \geq
    z_{0}$ implies
    \begin{equation}
        \sum_{m+n \geq z} e^{-C (m+n)^{\nu}} \leq \frac{2 }{C \nu }
        z^{2-\nu} e^{-C z^{\nu}}.
    \end{equation}
\end{lemma}

\begin{proof}
    First notice that
    \begin{equation}
        \sum_{m+n \geq z} e^{-C (m+n)^{\nu}} = \sum_{t \geq z} (t+1)
        e^{- C t^{\nu}} \leq \int_{z}^\infty (t+1) e^{-C t^{\nu}} dt.
        \nonumber
    \end{equation}
    When $t \geq z$ and $z$ is large enough,
    we have
    \begin{eqnarray*}
        \int_{z}^\infty (t+1) e^{-C t^{\nu}} dt &\leq& \frac{2}{C\nu}
        \int_{z}^{\infty} \paren{ C \nu t - (2-\nu) t^{1-\nu} } e^{-C
          t^{\nu}} dt \\
        &\leq &\frac{2}{C\nu} z^{2-\nu} e^{-Cz^\nu}
    \end{eqnarray*}
    which is what we needed to prove.
\end{proof}


\section{Density matrix estimation}
\label{sec.dens.mat}

\noindent The aim of this part is to estimate the density matrix $\rho$ in the
Fock basis directely from the data $(Y_i,\Phi_i)_{i=1,\ldots,n}$. We
show that for $0 < \eta \leq 1/2$ it is still possible to estimate the
density matrix with an error of estimation tending to $0$ as $n$ tends
to infinity (Theorem~\ref{theo:bcpdebruit}). In both cases ($\eta > \frac12$
and $\eta \leq \frac12$), we construct an estimator of the density
matrix $(\rho_{j,k})_{j,k\leq N-1}$ from a sample of QHT
data. We give theoretical results for our estimator when the quantum
state $\rho$ is in the class of density matrix with decreasing
elements defined in (\ref{eq.classcoeff}).

\subsection{Pattern functions}

\noindent The matrix elements $\rho_{j,k}$ of the state $\rho$ in the Fock
basis~(\ref{eq:2}) can be expressed as kernel integrals: for all $j,k
\in \N$,
\begin{eqnarray}
    \label{rhojk}
    \rho_{j,k}=\frac 1\pi\int \int_0^\pi
    p_\rho(x,\phi)f_{j,k}(x)e^{-i(k-j)\phi}d\phi dx
\end{eqnarray}
where $f_{j,k} = f_{k,j}$ are bounded real functions called
\textit{pattern functions} in quantum homodyne literature. A concrete
expression for their Fourier transform using Laguerre polynomials was
found in \cite{Richter1}: for $j \geq k$,
\begin{eqnarray}
    \tilde{f}_{k,j}(t) &=& 2 \pi^2 \abs{t} \widetilde{W_{j,k}}(t,0)
    \nonumber
    \\   \label{pattern}
    &=&\pi(-i)^{j-k}\sqrt{\frac{2^{k-j} k!}{j!}}|t|
    t^{j-k}e^{-\frac{t^2}{4}}L^{j-k}_k(\frac{t^2}{2}).
\end{eqnarray}
where $\tilde{f}_{k,j}$ denotes the Fourier transform of the Pattern function $f_{k,j}$.

\noindent Let us state the lemmata which are used to prove upper bounds in Propositions~\ref{prop:3},~\ref{prop:4} and~\ref{prop:5}.

\begin{lemma}
    \label{lm:4}
    There exist constants $C_2$, $C_\infty$ such that
    \begin{center}
        $\sum_{j+k=0}^{N} \norm{2}{f_{k,j}}^2 \leq C_2 N^{\frac{17}{6}}$ and $\quad\sum_{j + k = 0}^{N} \norm{\infty}{f_{k,j}}^2
        \leq C_\infty N^{\frac{10}{3}}$.
    \end{center}
\end{lemma}
\noindent This is a slight improvement over~\cite[Lemma 1]{Artiles&Gill&Guta}.

\begin{proof}
    By symmetry we can restrict the sum to $j \geq k$.  For fixed $k$
    and $j$ we have
    \begin{eqnarray*}
        \norm{2}{\tilde{f}_{k,j}}^2 &=& \int_{\abs{t}<2s}
        \abs{\tilde{f}_{k,j}(t)}^2 dt
        + \int_{\abs{t}>2s} \abs{\tilde{f}_{k,j}(t)}^2 dt
    \end{eqnarray*}
    (with $s = \sqrt{k+j+1}$).  Because of Lemma~\ref{lemm:estimeLaguerre}, it is clear that the second
    integral is negligible in front of the first one, which we simply bound by $4 s \norm{\infty}{\tilde{f}_{k,j}}^2$.

    \noindent In view of (\ref{pattern}), the main result in~\cite{Krasikov:2007uq} can be rewritten as follows: if $k \geq
    35$ and $j-k \geq 24$, then
    \begin{equation}
        \label{eq:6}
        \norm{\infty}{\tilde{f}_{k,j}}^2 \leq 2888 \pi^2
        (j+1)^{\frac12} k^{-\frac16}.
    \end{equation}
    In consequence, for these values of $k$ and $j$,
    \begin{equation}
        \label{eq:3}
        \norm{2}{\tilde{f}_{k,j}}^2 \leq C (j k^{-\frac16} +
        j^{\frac12} k^{\frac13} ).
    \end{equation}

    \noindent On the other hand, a classical bound on Laguerre polynomials found in~\cite{Szego:1959uq} yields that, for fixed      values of $j-k$, $\norm{\infty}{\tilde{f}_{k,j}}^2 \leq C k^{\frac13}$, hence for all $k \geq 35$ and $j-k < 24$,
    \begin{equation}
        \label{eq:4}
        \norm{2}{\tilde{f}_{k,j}}^2 \leq C (j^{\frac12}k^{\frac13} +
        k^{\frac56}).
    \end{equation}

    \noindent When $k < 35$, we can use another result in~\cite{Krasikov:2005db} which gives $\norm{\infty}{\tilde{f}_{k,j}}^2      \leq C k^{\frac16}j^{\frac12}$ independently of $j-k$, thus
    \begin{equation}
        \label{eq:5}
        \norm{2}{\tilde{f}_{k,j}}^2 \leq C j.
    \end{equation}
    Comparing~(\ref{eq:3}),~(\ref{eq:4}) and~(\ref{eq:5}) we see that when $N$ is large enough, in the sum over $0 \leq j, k        \leq N$, the terms $k \geq 35$, $j-k \geq 24$ dominate and~(\ref{eq:3}) yields the first inequality.

    \noindent The second inequality is obtained by doing a similar computation, starting with $\norm{\infty}{f_{j,k}} \leq
    \norm{1}{\tilde{f}_{j,k}}$ and using~(\ref{eq:6}) to bound
    \begin{equation*}
        \norm{1}{\tilde{f}_{j,k}}^2 \leq C(j^{\frac32} k^{-\frac16}
        + j^{\frac12} k^{\frac56})
    \end{equation*}
    when $k \geq 35$ and $j-k \geq 24$.
\end{proof}

\noindent In the presence of noise, it is necessary to adapt the pattern
functions as follows.  From now on, we shall use the notation
\fbox{$\gamma \deq \frac{1-\eta}{4\eta}$}. When $\frac12 < \eta \leq
1$, we denote by $f^\eta_{k,j}$ the function which has the following
Fourier transform:
\begin{equation}
    \label{eq:patterneta}
    \tilde{f}^\eta_{k,j}(t) \deq \tilde{f}_{k,j}(t) e^{\gamma t^2}.
\end{equation}
When $0< \eta \leq \frac12$, we introduce a cut-off parameter $\delta >
0$ and define $f^{\eta,\delta}_{k,j}$ via its Fourier transform:
\begin{equation}
    \label{eq:patternetadelta}
    \tilde{f}^{\eta,\delta}_{k,j}(t) \deq
    \tilde{f}_{k,j}(t)e^{\gamma t^2} \mathbb{I}\paren{\abs{t}\leq
    \frac1\delta}.
\end{equation}

\noindent Then we compute bounds on these pattern functions.

\begin{lemma}
    \label{lm:5}
    For $1>\eta>1/2$,  there exist constants $C_2^\eta$ and
    $C_\infty^\eta$ such that
    \begin{center}
        $\sum_{j + k = 0}^{N} \norm{2}{f^\eta_{k,j}}^2 \leq
        C_2^\eta N^{\frac56} e^{8\gamma N}$  and $\quad \sum_{j + k=0}^{N} \norm{\infty}{f^\eta_{k,j}}^2
        \leq C_\infty^\eta N^{\frac{1}{3}} e^{8\gamma N}$.
    \end{center}
\end{lemma}

\begin{proof}
    The proof is similar to the previous one and we skip some
    details. Once again we assume $j \geq k$ and write
    \begin{eqnarray*}
        \norm{2}{\tilde{f}_{k,j}^\eta}^2 &=& \int_{\abs{t}<2s}
        \abs{\tilde{f}_{k,j}(t)}^2 e^{2\gamma t^2}    dt
        + \int_{\abs{t}>2s} \abs{\tilde{f}_{k,j}(t)}^2
        e^{2\gamma t^2} dt
    \end{eqnarray*}
    (where $s = \sqrt{k+j+1}$).  Because of
    Lemma~\ref{lemm:estimeLaguerre}, the second integral is of the
    same order as the first one, which we bound by
    \begin{equation*}
        \norm{\infty}{\tilde{f}_{k,j}}^2 \int_{\abs{t} < 2s}
        e^{2\gamma t^2}  dt \leq C
        \norm{\infty}{\tilde{f}_{k,j}}^2
        s^{-1} e^{8\gamma s^2}.
    \end{equation*}
    In the sum we are considering the terms $k \geq 35$ and $j-k \geq
    24$ are dominant and, once again thanks to~(\ref{eq:6}),
    remembering that $s = \sqrt{j+k+1}$,
    \begin{equation*}
        \norm{2}{\tilde{f}_{k,j}^\eta}^2 \leq C
        k^{-\frac16} e^{8\gamma (j+k)}
    \end{equation*}
    hence the first inequality.

    \noindent The second inequality is, in the same fashion, based on
    \begin{eqnarray*}
        \norm{\infty}{{f}_{k,j}^\eta}^2 \leq
        \norm{1}{\tilde{f}_{k,j}^\eta}^2 &\leq& C \paren{j^{\frac14}
          k^{-\frac1{12}} \int_{\abs{t} < 2 s}
          e^{\gamma t^2} dt}^2 \\
        &\leq& C j^{-\frac12} k^{-\frac16} e^{8\gamma(j+k)}
    \end{eqnarray*}
    when $k \geq 35$ and $j-k \geq 24$, and the bound on the sum
    readily follows.
\end{proof}

\subsection{Estimation procedure}

\noindent For $N \deq N(n) \rightarrow \infty$ and $\delta \deq \delta(n)
\rightarrow 0$, let us define our estimator of $\rho_{j,k}$ for $0
\leq j+k \leq N-1$ by
\begin{equation}
    \label{estrhojk}
    \hat{\rho}^\eta_{j,k} \deq \frac{1}{n}
    \sum_{\ell=1}^{n}G_{j,k}\paren{\frac{Y_\ell}{\sqrt\eta},\Phi_\ell},
\end{equation}
where  \begin{equation*}
     G_{j,k}(x,\phi) \deq\left\{ \begin{array}{ll}
                                 f^\eta_{j,k}(x) e^{-i(j-k)\phi} &{\rm \ if \ } \frac12 < \eta \leq 1 \\
                                 f^{\eta,\delta}_{j,k}(x) e^{-i(j-k)\phi} &{\rm \ if \ } 0 < \eta \leq\frac12.
                                       \end{array} 
                               \right. 
	       \end{equation*}
using the pattern functions defined in~(\ref{eq:patterneta}) and~(\ref{eq:patternetadelta}). 
We assume that the density matrix $\rho$ belongs to the class $\mathcal{R}(B,r)$
defined in (\ref{eq.classcoeff}). In order to evaluate the performance
of our estimators we take the $\mathbb{L}_2$ distance on the space of
density matrices $\left\|\tau-\rho\right\|^2_2 \deq \Tr(|\tau-\rho|^2)
= \sum_{j,k=0}^\infty |\tau_{j,k}-\rho_{j,k}|^2$. We consider the mean
integrated square error (MISE) and split it into a troncature bias
term $b_1^2(n)$, a regularization bias terms $b_2^2(n)$ and a variance
term $\sigma^2(n)$.
\begin{eqnarray*}
    E \paren{ \sum_{j,k=0}^\infty \left|\hat{\rho}^\eta_{j,k}-
          \rho_{j,k}\right|^2 } &=& \sum_{j+k \geq N}
    \left|\rho_{j,k}\right|^2+\sum_{j+k=0}^{N-1}
    \left|E[\hat{\rho}^\eta_{j,k}]-
        \rho_{j,k}\right|^2 \\
    &&+ \sum_{j+k=0}^{N-1} E\left|\hat{\rho}^\eta_{j,k} -
        E[\hat{\rho}^\eta_{j,k}]\right|^2 \\
    &\eqd& b_1^2(n) + b_2^2(n) + \sigma^2(n).
\end{eqnarray*}
The following propositions give upper bounds for $b_1^2(n)$,
$b_2^2(n)$ and $\sigma^2(n)$ in the different cases $\eta = 1$, $1/2 <
\eta < 1$ or $0 < \eta \leq 1/2$ and $r=2$ or $0<r<2$. Their proofs
are defered to Appendix~\ref{sec.proofs}.

\begin{proposition}
    \label{prop:3}
    Let $\hat{\rho}^\eta_{j,k}$ be the estimator defined by
    (\ref{estrhojk}), for $0<\eta<1$, with $\delta\rightarrow 0$ and
    $N\rightarrow\infty$ as $n\rightarrow \infty$, then for all $B>0$
    and $0 < r \leq 2$,
    \begin{equation}
        \label{biais}
        \sup_{\rho\in\mathcal{R}(B,r)} b_1^2(n) \leq
        c_1 N^{2-r/2} e^{-2 B N^{r/2}}
    \end{equation}
    where $c_1$ is a positive constant depending on $B$ and $r$.
\end{proposition}

\begin{proposition}
    \label{prop:4}
    Let $\hat{\rho}^\eta_{j,k}$ be the estimator defined by
    (\ref{estrhojk}), for $0 < \eta \leq 1/2$, with $N \rightarrow
    \infty$ as $n \to \infty$ and $1 / \delta \geq 2\sqrt{N}$.
    In the case $r = 2$, for $\beta := B /(1+\sqrt{B})^2$ there exists $c_2$, 
	while in the case $0 < r < 2$, for any $\beta < B$ there exists $c_2$
    and $n_0$ such that for $n \geq n_0$: 
    \begin{equation}
        \label{biaisbis}
        \sup_{\rho\in\mathcal{R}(B,r)}b_2^2(n) \leq
        c_2 N^{2} \delta^{4r-12} e^{-\frac{2\beta
          }{(2\delta)^r}-\frac 12
          \left(\frac{1}{\delta}-2\sqrt{N}\right)^2}.
    \end{equation}
\end{proposition}
\noindent Note that for $1/2 < \eta \leq 1$ we have $b_2(n)=0$ for all $0 < r
\leq 2$ ($\hat \rho^\eta_{j,k}$ is unbiased).

\begin{proposition}
    \label{prop:5}
    For $\hat{\rho}^\eta_{j,k}$ the estimator defined by
    (\ref{estrhojk}),
    \begin{eqnarray}
        \label{eq:varbcpdebruit}
        \sup_{\rho\in\mathcal{R}(B,r)}\sigma^2(n) \leq c_3
        \frac{\delta N^{17/6}}{n} e^{\frac{2\gamma}{\delta^2}}
              &&\quad{\rm if \ } 0<\eta\leq 1/2\\
        \label{eq:varpeudebruit}
        \sup_{\rho\in\mathcal{R}(B,r)}\sigma^2(n) \leq c_3'
        \frac{N^{1/3}}{n}e^{8\gamma N}&&\quad{\rm if \ }1/2<\eta<1 \\
        \label{eq:variancesansbruit}
        \sup_{\rho\in\mathcal{R}(B,r)}\sigma^2(n)
        \leq c_3'' \frac{N^{\frac{17}{6}}}{n}&& \quad{\rm if \ } \eta = 1
    \end{eqnarray}
    where $c_3$, $c_3'$ are positive constants depending on $\eta$.
\end{proposition}

\noindent We measure the accuracy of $\hat{\rho}^\eta_{j,k}$ by the
maximal risk over the class $\mathcal{R}(B,r)$
\begin{eqnarray}
    \label{risk1}
    \limsup_{n\to\infty}
    \sup_{\rho\in\mathcal{R}(B,r)}\varphi_n^{-2}E\left(\sum_{j,k=0}^\infty
        \left|\hat{\rho}^\eta_{j,k}- \rho_{j,k}\right|^2\right)\leq
    C_0.
\end{eqnarray}
where $C_0$ is a positive constant and $\varphi_n^{2}$ is a sequence
which tends to 0 when $n\rightarrow\infty$ and it is the rate of
convergence. Cases $\eta = 1$ (no noise), $\frac12<\eta<1$ (weak
noise) and $0<\eta\leq\frac12$ (strong noise) are studied respectively in
Theorems~\ref{theo:sansbruit},~\ref{theo:peudebruit}
and~\ref{theo:bcpdebruit}.

\begin{theorem}
    \label{theo:sansbruit}
    When $\eta = 1$, the estimator defined in (\ref{estrhojk}) for the
    model (\ref{noisy.data}), where the unknown state belongs to the
    class $\mathcal{R}(B,r)$, satisfies the upper bound (\ref{risk1})
    with
    \begin{equation*}
        \varphi_n^2 = \log(n)^{\frac{17}{3 r}} n^{-1}
    \end{equation*}
    obtained by taking $N(n) \deq \paren{\frac{\log(n)}{2
        B}}^{\frac2r}$.
\end{theorem}

\begin{proof}
    With the proposed $N(n)$ one checks that the bias~(\ref{biais}) is
    smaller than the variance~(\ref{eq:variancesansbruit}) which is
    bounded by a constant times ${\log(n)^{\frac{17}{3 r}}}{n^{-1}}$.
\end{proof}

\begin{theorem}
    \label{theo:peudebruit}
    When $\frac12<\eta<1$, the estimator defined in (\ref{estrhojk})
    for the model (\ref{noisy.data}), where the unknown state belongs
    to the class $\mathcal{R}(B,r)$, satisfies the upper bound
    (\ref{risk1}) with
    \begin{itemize}
      \item For $r=2$,
        \begin{equation*}
            \varphi_n^2 = \log(n)^{ \frac{12\gamma+B}{3(4\gamma +B)}} 
				n^{-\frac{B}{4\gamma+B}}
        \end{equation*}
    with $N(n) \deq \frac{\log(n)}{2(4\gamma+B)}\left( 1+\frac 23 \frac{\log(\log n)}{\log(n)}\right)$.
      \item For $0<r<2$,
        \begin{equation*}
            \varphi_n^2 = \log(n)^{2-r/2}
            e^{-2B{N(n)}^{r/2}}
        \end{equation*}
        where $N(n)$ is the solution of the equation $8 \gamma N + 2 B
        N^{r/2} = \log(n)$.
    \end{itemize}
   
    \noindent In that case we have $N(n) = \frac{1}{8\gamma} \log(n) -
    \frac{2B}{(8\gamma)^{1+r/2}} \log(n)^{r/2} + o(\log(n)^{r/2})$.
\end{theorem}

\begin{proof}
    When $r=2$, the proposed $N(n)$ ensures that the
    variance~(\ref{eq:varpeudebruit}) is equivalent to the
    bias~(\ref{biais}), which is bounded by a constant times 
	$\log(n)^{\frac{12\gamma +B}{3(4\gamma +B)}}
    n^{-\frac{B}{4\gamma+B}}$.

    \noindent When $0 < r < 2$, the proposed $N(n)$ makes the
    variance~(\ref{eq:varpeudebruit}) bounded by a constant times
    $e^{-2B{N(n)}^{r/2}}$, which is smaller than the bias, the latter
    being bounded by a constant times $N(n)^{2-r/2}
    e^{-2B{N(n)}^{r/2}}$.

     \noindent The asymptotic expansion of $N(n)$ is a standard consequence of
    its definition by the equation $8 \gamma N + 2 B N^{r/2} =
    \log(n)$.
\end{proof}

\begin{theorem}
    \label{theo:bcpdebruit}
    When $0 < \eta \leq \frac12$, the estimator defined in
    (\ref{estrhojk}) for the model (\ref{noisy.data}), where the
    unknown state belongs to the class $\mathcal{R}(B,r)$, satisfies
    the upper bound (\ref{risk1}) with
    \begin{equation*}
        \varphi_n^2 = N^{2-r/2}e^{-2 B N^{r/2}}
    \end{equation*}
    where $N$ and $\delta$ are solutions of the system
   \begin{equation}
     \left\{ \begin{array}{ll}
            \frac{2 \beta}{(2\delta)^r} + \frac12
            (\frac1\delta-2\sqrt{N})^2 + \frac{2 \gamma}{\delta^2}
            = \log(n)\\
            \frac{2\beta}{(2\delta)^r} + \frac12
            (\frac1\delta-2\sqrt{N})^2 - 2B N^{r/2} = (\log\log(n))^2
                                       \end{array} 
                               \right. 
	    \label{eq:7}
\end{equation}
    for arbitrary $\beta < B$ in the case $0 < r < 2$ or
  \begin{equation}
     \left\{ \begin{array}{ll}
            \frac{\beta + 4 \gamma}{2 \delta^2} + \frac12
            (\frac1\delta - 2\sqrt{N})^2 - \frac53 \log(N) = \log(n) \\
            \frac{\beta}{2\delta^2} + \frac12 (\frac1\delta -
            2\sqrt{N})^2 - 2BN - 3 \log(N) = 0
                                       \end{array} 
                               \right. 
	   \label{eq:8}
\end{equation}
    with $\beta \deq \frac{B}{(1+\sqrt{B})^2}$ in the case $r = 2$.
\end{theorem}

\noindent Theses bounds are optimal in the sense that~(\ref{eq:7})
and~(\ref{eq:8}) are obtained by minimizing the sum of the
bounds~(\ref{biais}),~(\ref{biaisbis}) and~(\ref{eq:varbcpdebruit}).

\begin{proof}
    We use the standard notations $a(n) \sim b(n)$ if
    $\frac{a(n)}{b(n)} \to 1$ and $a(n) \approx b(n)$ if there exists
    a constant $M < \infty$ such that $\frac{1}{M} \leq
    \frac{a(n)}{b(n)} \leq M$ for all $n$.

    \noindent Let us first examine the case $0 < r < 2$. Remark that the
    left-hand term of the second equation in~(\ref{eq:7}) is strictly
    negative when $1/\delta = 2\sqrt{N}$ and increases to $\infty$
    with $1/\delta$. This proves that the solution satisfies $1/\delta
    > 2\sqrt{N}$ and that Proposition~\ref{prop:4}
    applies. Furthermore, if we suppose that
    $\frac{1/\delta}{\sqrt{N}}$ is unbounded when $n \to \infty$, then
    (up to taking a subsequence) by the first equation
    $\frac{\frac12+2\gamma}{\delta^2} \sim \log(n)$ whereas, by
    subtracting the two, $\frac{2 \gamma}{\delta^2} \sim \log(n)$,
    which is contradictory. So $1/\delta \approx \sqrt{N}$ and we
    deduce that $N \approx \log(n)$.  Then~(\ref{biaisbis}) yields
    \begin{equation*}
        \log\paren{\frac{b_2^2(n)}{N^{2-r/2} e^{-2BN^{r/2}}}} \leq
        (4r-12) \log(\delta) + \frac{r}{2} \log(N) - (\log\log(n))^2
        \to -\infty
    \end{equation*}
    whereas~(\ref{eq:varbcpdebruit}) gives
    \begin{equation*}
        \log\paren{\frac{\sigma^2(n)}{N^{2-r/2} e^{-2BN^{r/2}}}} \leq
        \log(\delta) + (\frac56+\frac{r}{2}) \log(N) - (\log\log(n))^2
        \to -\infty.
    \end{equation*}
    We see that the dominant term is the bound~(\ref{biais}) on
    $b_1^2(n)$, hence the result.

    \noindent When $r = 2$, the same reasoning as above yields $1/\delta >
    2\sqrt{N}$, $1/\delta \approx \sqrt{N}$ and $N \approx
    \log(n)$. Then the right-hand side of~(\ref{biaisbis})
    and~(\ref{eq:varbcpdebruit}) are of the same order as $N
    e^{-2BN}$, which is the bound~(\ref{biais}) on $b_1^2(n)$.
\end{proof}

\section{Wigner function estimation}
\label{sec.Wigner}

\subsection{Kernel estimator}

\noindent We describe now the direct estimation method for the Wigner function.
For the problem of estimating a probability density
$f:\mathbb{R}^{2}\to \mathbb{R}$ directly from data $(X_\ell,
\Phi_\ell)$ with density $\mathcal{R}[f]$ we refer to the literature
on X-ray tomography and PET, studied by \cite{Vardi, Korostelev&Tsybakov, 
Johnstone&Silverman, Cavalier} and 
many other references therein.  In the context of tomography of
bounded objects with noisy observations \cite{Goldenshluger&Spokoiny} solved the problem of
estimating the borders of the object (the support).  The estimation of
a quadratic functional of the Wigner function has been treated in
\cite{Meziani}. For the problem of Wigner function estimation when no
noise is present, we mention the work by \cite{Guta&Artiles}.  They use a kernel estimator and compute
sharp minimax results over a class of Wigner functions characterised
by their smoothness.  In a more recent
paper~\cite{Butucea&Guta&Artiles}, Butucea, Gu\c t\u a and Artiles
treated the noisy problem for the pointwise estimation of $W_{\rho}$;
however the functions needed to prove minimax optimality there do not
belong to the class of Wigner functions that we consider here.

\noindent In this chapter, as in~\cite{Butucea&Guta&Artiles}, we modify the usual
tomography kernel in order to take into account the additive noise on
the observations and construct a kernel $K_h^\eta$ which performs both
deconvolution and inverse Radon transform on our data, asymptotically.
Let us define the estimator:
\begin{equation}\label{estimator}
    \widehat{W}^\eta_{h}(q,p) = \frac{1}{\pi n} \sum_{\ell=1}^n
    K_h^\eta \left(q \cos \Phi_\ell + p \sin \Phi_\ell
        -\frac{Y_\ell}{\sqrt{\eta}} \right),
\end{equation}
where $0<\eta<1$ is a fixed parameter, and the kernel is defined by
\begin{equation}\label{noisy.reg.operator}
    K_h^\eta (u) = \frac{1}{4 \pi}\int_{-1/h}^{1/h}
    \frac{\exp(-iut)|t|}{\widetilde N^\eta (t/\sqrt{\eta})} dt, \quad
    \widetilde K_h^\eta\left(t \right) = \frac{1}{2} \frac{|t|}{\widetilde
      N^\eta (t/\sqrt{\eta})}I(|t|\leq 1/h),
\end{equation}
and $h>0$ tends to $0$ when $n \to \infty$ in a proper way to be
chosen later. For simplicity, let us denote $z=(q,p)$ and $ [z,\phi]
=q \cos \phi + p \sin \phi$, then the estimator can be written:
\begin{equation*}
    \widehat{W}^\eta_{h}(z) = \frac{1}{\pi n} \sum_{\ell=1}^n
    K_h^\eta \left([ z,\Phi_\ell ] -\frac{Y_\ell}{\sqrt{\eta}}\right).
\end{equation*}

\noindent This is a one-step procedure for treating two successive inverse
problems.  The main difference with the noiseless problem treated 
by~\cite{Guta&Artiles} is that the deconvolution
is more `difficult' than the inverse Radon transform.  In the
literature on inverse problems, this problem would be qualified as
severely ill-posed, meaning that the noise is dramatically
(exponentially) smooth and makes the estimation problem much harder.

%
\subsection{$\mathbb{L}_2$ risk estimation}\label{sec.l2.est}
%
\noindent We establish next the rates of estimation of $W_\rho$ from i.i.d.
observations $(Y_\ell,\Phi_\ell), \, \ell=1,\ldots,n$ when the quality
of estimation is measured in $\mathbb{L}_2$ distance.  In the
literature, $\mathbb{L}_2$ tomography is usually performed for
boundedly supported functions, see~\cite{Korostelev&Tsybakov} 
and~\cite{Johnstone&Silverman}.  However, most Wigner function
do not have a bounded support!  Instead, we use the fact that Wigner
functions in the class $\mathcal{R}(B,r)$ decrease very fast and show
that a properly truncated estimator attains the rates we may expect
from the statistical problem of deconvolution in presence of
tomography.  Thus, we modify the estimator by truncating it over a
disc with increasing radius, as $n \to \infty$.  Let us denote
$$
D(s_n) = \left\{z=(q,p) \in \mathbb{R}_2: \|z\| \leq s_n \right\},
$$
where $s_n \to \infty$ as $n \to \infty$ will be defined in
Theorem~\ref{minimax.l2.bounds}.  Let now
\begin{equation}\label{truncated.estimator}
    \widehat W^{\eta, *}_{h,n}(z) = \widehat{W} ^\eta_{h,n} (z) I_{D(s_n)}(z).
\end{equation}
From now on, we will denote for any function $f$,
$$
\|f\|^2_{D(s_n)}= \int_{D(s_n)} f^2 (z) dz ,
$$
and by $\overline D(s_n)$ the complementary set of $D(s_n)$ in $\r^2$.
Then,
\begin{eqnarray*}
    E\left[\left\|\widehat W_{h}^{\eta,*} - W_{\rho} \right\|_2^2\right]
    &=& E\left[\left\|\widehat W_{h}^\eta - W_{\rho} \right\|_{D(s_n)}^2
    \right] + \|W_\rho\|^2_{\overline D(s_n)} \\
    &=& E\left[\left\|\widehat W_{h}^\eta -
            E\left[\widehat{W}_{h}^\eta\right] \right\|_{D(s_n)}^2 \right]
    +\left\|E\left[\widehat W_{h}^\eta\right] - W_{\rho}
    \right\|_{D(s_n)}^2 \\
    && +\|W_\rho\|^2_{\overline D(s_n)}.
\end{eqnarray*}
\noindent When replacing the $\mathbb{L}_2$ norm with the above restricted
integral, the upper bound of the bias of the estimator is unchanged,
whereas the variance part is infinitely larger 
than the deconvolution variance in~\cite{Butucea&Tsybakov}.  
As the bias is dominating over the
variance in this setup, we can still choose a suitable sequence $s_n$
so that the same bandwidth is optimal associated to the same optimal
rate, provided that $W_\rho$ decreases fast enough asymptotically.
\noindent The following proposition gives upper bounds for the three components
of the $\mathbb{L}_2$ risk uniformly over the class $\mathcal R(B,r)$.
\begin{proposition}
    \label{bias.varL2}
    Let $(Y_\ell,\Phi_\ell),\, \ell = 1,\ldots,n$ be i.i.d. data
    coming from the model~ ($\ref{noisy.data}$) and let $\widehat
    W_{h}^\eta$ be an estimator (with $h \to 0$ as $n \to \infty$)
    of the underlying Wigner function $W_\rho$.  We suppose $W_\rho$
    lies in the class $\mathcal R(B,r)$, with $B>0$ and $0<r\leq 2$.
    Then, for $s_n \to \infty$ as $n \to \infty$ and $n$ large enough,
    \begin{eqnarray*}
        \sup_{\rho \in \mathcal R(B,r)}
        \| W^\rho \|^2_{\bar D(s_n)}
        &\leq &
        C_1 s_n^{10-3r} e^{-2\beta s_n^r}   ,\\
[2mm]
        \sup_{\rho \in \mathcal R(B,r)}\left\|E[\widehat
            W^\eta_{h}] -W_\rho \right\|^2_{D(s_n)}
        & \leq &
        C_{2 } h^{3r-10} e^{- \frac{2^{1-r} \beta}{ h^r}}   \\
[2mm]
        \sup_{\rho \in \mathcal R(B,r)}
        E\left[\left\|\widehat W_{h,n}^\eta - E
                \left[\widehat W_{h,n}^\eta \right]\right\|_{D(s_n)}^2 \right]
        & \leq & C_3
        \frac{s_n^2}{ n h} \exp \left(\frac{2 \gamma}{h^2} \right),
    \end{eqnarray*}
    where $\beta < B $ is defined in Proposition~\ref{prop:popetown} for $0<r<2$
and $\beta = B/(1+\sqrt{B})^2$ for $r=2$,
    $\gamma = (1-\eta)/ (4 \eta)>0$, $C_1,\, C_2,\,C_3$ are positive constants, 
$C_1,\,C_2,$  depending on $\beta, \, B,\, r$ and $C_3$ depending only on $\eta$.
\end{proposition}
\noindent We measure the accuracy of $\widehat W_{h}^{\eta,*} $ by the
maximal risk over the class $\mathcal{R}(B,r)$
\begin{eqnarray}
    \label{risk2}
     \limsup_{n \to \infty} \sup_{\rho \in \mathcal R(B,r)}
    E\left[\left\| \widehat W_{h}^{\eta,*} - W_\rho\right\|^2\right]
    \varphi_n^{-2}(\l_2) \leq C.
\end{eqnarray}
where $C$ is a positive constant and $\varphi_n^{2}$ is a sequence
which tends to 0 when $n\rightarrow\infty$ and it is the rate of
convergence. 

\noindent In the following Theorem we see the phenomenon which was noticed
already: deconvolution with Gaussian type noise is a much harder
problem than inverse Radon transform (the tomography part).

\begin{theorem}
    \label{minimax.l2.bounds}
    Let $B >0$, $0 < r \leq 2$ and $(Y_\ell,\Phi_\ell),\, \ell =
    1,\ldots,n$ be i.i.d.  data coming from the
    model~($\ref{noisy.data}$).  Then $\widehat W_{h}^{\eta,*}$
    defined in (\ref{truncated.estimator}) with kernel $K_h^\eta$ in
    (\ref{noisy.reg.operator}) satisfies the upper bound
    (\ref{risk2}) with
    \begin{itemize}
      \item For $r=2$, put $\beta = B/(1+\sqrt{B})^2$
        \begin{equation*}
            \varphi_n^2 = (\log n)^{\frac{16\gamma+3\beta}{8\gamma+2\beta}} n^{-\frac{\beta}{4\gamma+\beta}},
        \end{equation*}
        with $s_n=(h)^{-1}$ and $h=\left(\frac{2}{4\gamma+\beta} 
      \log n+\frac{1}{4\gamma+\beta}\log( \log n)\right)^{-1/2}$.
      \item For $0<r<2$ and $\beta < B $ defined in Proposition~\ref{prop:popetown},
        \begin{equation*}
            \varphi_n^2 = h^{3r-10} \exp\left(- \frac{ 2^{1-r}\beta}{h^r} \right),
        \end{equation*}
        where $s_n=1/h$ and $h$ is the solution of the equation 
        $$\frac{2^{1-r}\beta}{h^r} + \frac{2\gamma}{ h^2}
        =\log n -(\log \log n)^2.$$
    \end{itemize}
\end{theorem}

\noindent \begin{proof}[Sketch of proof of the upper bounds] By
    Proposition~\ref{bias.varL2}, we get
    \begin{eqnarray*}
        \sup_{W_\rho \in \mathcal R(B,r)} E\left[\left\| \widehat
                W_{h}^\eta - W_\rho\right\|^2\right]
        & \leq & C_1 s_n^{10-3r} e^{-2 \beta s_n^r}
       +C_2 h^{3r-10} \exp\left(- \frac{2
       \beta}{(2h)^r}\right)\\
       && + \frac{C_3 s_n^2}{nh} \exp\left(\frac{2\gamma}{h^2} \right).\\
              &=:& A_1+A_2+A_3
    \end{eqnarray*}
     For $0<r<2$ and by taking derivatives with respect to $h$ and $s_n$, we
obtain that the optimal choice verifies the following equations:
\begin{eqnarray*}
2\beta s_n^r+\frac{2 \gamma}{h^2} &=& \log(n) + \log(h s_n^{2(4-r)})\\
\frac{2^{1-r}\beta}{h^r}+\frac{2 \gamma}{h^2} &=& \log(n) + \log(h^{2r-7} s_n^{-2}).
\end{eqnarray*}
We notice therefore that $A_2$ is dominating over $A_3$, which is dominating over $A_1$.
The proposed $(s_n,h)$ ensure that the term $A_2$ is still the dominating term and 
gives the rate of convergence.

\noindent The case $r=2$ is treated similarly, by taking derivatives we notice that the term $A_2$ and the term $A_3$ are       of the same order and that the term $A_1$ is smaller than the others. \hfill \end{proof}

\section{Appendix}
\label{sec.proofs}

\subsection{Proof of Proposition \ref{prop:popetown}}

\noindent Let $\phi(z) \deq (z - \sqrt{\beta} z^{r/2})^2 - 1$.  Since $r <
    2$, for $z$ larger than a certain $z_{0}$ (which depends only on
    $\beta$, $B$ and $r$), it is true that $\phi(z) \geq
    \paren{\frac{\beta}{B}}^{2/r} z^2$.  It follows that
    \begin{equation}
        \label{eq:whereisthisgoing}
        e^{-B \phi(z)^{r/2}} \leq e^{- \beta z^r}
    \end{equation}

    \noindent If $m+n \leq \phi(z)$, then $s \leq \sqrt{1+\phi(z)}$ and $z - s
    \geq z - \sqrt{1+\phi(z)} = \sqrt{\beta} z^{r/2}$.
    By~(\ref{eq:estimeLaguerre}), this means that $l_{m,n}(z) \leq \frac1\pi
    e^{-\beta z^r}$.  So
    \begin{equation}
        \label{eq:majsum1}
        \sum_{m+n \leq \phi(z)} \abs{\rho_{m,n}} l_{m,n}(z) \leq A
        e^{-\beta z^r}
    \end{equation}
    for $A \deq \frac1\pi \sum_{m,n} e^{-B (m+n)^{r/2}}$.
    
    \noindent On the other hand, using Lemma~\ref{lemm:magicjohnson} with $\nu
    \deq r/2$, if $\phi(z) \geq z_{0}$,
    \begin{eqnarray}
        \sum_{m+n \geq \phi(z)} \abs{\rho_{m,n} } l_{m,n}(z) &\leq&
        \frac{4}{\pi B r} \phi(z)^{2-r/2} e^{-B \phi(z)^{r/2} } \nonumber
        \\
        \label{eq:majsum2}
        &\leq& \frac{4}{\pi B r} z^{4-r} e^{- \beta z^r}
    \end{eqnarray}
    by (\ref{eq:estimeLaguerre}) and (\ref{eq:whereisthisgoing}).
    Combining (\ref{eq:majsum1}) and (\ref{eq:majsum2}) yields the
    announced result.  The bound on $\Fourier{W_{\rho}}$ is then a
    direct consequence of (\ref{eq:FourierWmn}).
    
\subsection{Proof of Proposition \ref{prop:popetown2}}

    \noindent Let $\phi(z) \deq \theta z^2 - 1$, where $\theta \deq
    \frac{1}{(1+\sqrt{B})^2}$ is the solution in $(0,1)$ of
    $(1-\sqrt{\theta})^2 = B \theta$.

     \noindent When $m+n \leq \phi(z)$, then $s \leq \sqrt{\theta} z$ and $z - s
    \geq z (1- \sqrt{\theta}) = \sqrt{B \theta} z$.
    By~(\ref{eq:estimeLaguerre}), this means that $l_{m,n}(z) \leq \frac1\pi
    e^{-B \theta z^2}$.  So
    \begin{equation}
        \label{eq:majsumagain1}
        \sum_{m+n \leq \phi(z)} \abs{\rho_{m,n}} l_{m,n}(z) \leq A
        e^{-B\theta z^2}
    \end{equation}
    for $A \deq \frac1\pi \sum_{m,n} e^{ - B (m + n) }$.
    
    \noindent On the other hand, by Lemma~\ref{lemm:magicjohnson}, if $\phi(z)
    \geq z_{0}$,
    \begin{eqnarray}
        \sum_{m+n \geq \phi(z)} \abs{\rho_{m,n} } l_{m,n}(z) &\leq&
        \frac{2}{\pi B} \phi(z) e^{-B \phi(z)} \nonumber \\
        \label{eq:majsumagain2}
        &\leq& \frac{2 \theta e^{B}}{\pi B} z^{2} e^{- B \theta z^2}
    \end{eqnarray}
    by (\ref{eq:estimeLaguerre}) and (\ref{eq:whereisthisgoing}).
    Combining (\ref{eq:majsumagain1}) and (\ref{eq:majsumagain2})
    yields the announced result.  The bound on $\Fourier{W_{\rho}}$ is
    then a direct consequence of (\ref{eq:FourierWmn}).

\subsection{Proof of Proposition~\ref{prop:3}}
    
    \noindent By (\ref{eq.classcoeff}) the term $b_1^2(n)$ can be bounded as
    follows
    \begin{eqnarray*}
        b_1^2(n)&=&\sum_{j+k \geq N} \left|\rho_{j,k}\right|^2
        \leq  \sum_{j + k \geq N}\exp(-2B(j+k)^{r/2}).
    \end{eqnarray*}
    Compare to the double integral and change to polar coordinates to get
    \begin{eqnarray*}
        b_1^2(n)&\leq &  c_1 N^{2-r/2}\exp(- 2 B N^{r/2} ) .
    \end{eqnarray*}

\subsection{Proof of Proposition~\ref{prop:4}}
   
   \noindent To study the term $b_2^2(n)$, we denote
    \begin{eqnarray*}
        \mathcal{F}_1[p_\rho(\cdot|\phi)](t) \deq
        E_\rho[e^{itX}|\Phi=\phi] = \widetilde{W}_{\rho}(t\cos\phi,t\sin\phi),
    \end{eqnarray*}
    the Fourier transform with respect to the first variable.
    \begin{eqnarray*}
        E[\hat{\rho}^\eta_{j,k}] &=&
        E[G_{j,k}(\frac{Y}{\sqrt\eta},\Phi)] =
        E[f_{j,k}^{\eta,\delta}(\frac{Y}{\sqrt\eta})e^{-i(j-k)\Phi}]
        \nonumber\\
        &= &\frac{1}{\pi}\int_0^\pi e^{-i(j-k)\phi}\int
        f_{j,k}^{\eta,\delta}(y)\sqrt{\eta}p_\rho^\eta(y\sqrt{\eta}|\phi)dy
        d\phi\\
        &= &\frac{1}{\pi}\int_0^\pi e^{-i(j-k)\phi}\frac{1}{2\pi}\int
        \widetilde{f}_{j,k}^{\eta,\delta}(t)
        \mathcal{F}_1[\sqrt{\eta}p_\rho^\eta(\cdot\sqrt{\eta}|\phi)](t)dt
        d\phi\\
        &=& \frac{1}{\pi} \int_0^\pi e^{-i(j-k)\phi} \frac{1}{2\pi}
        \int_{\abs{t}\leq 1/\delta} \widetilde{f}_{j,k}(t) e^{\gamma
          t^2} \mathcal{F}_1[p_\rho(\cdot|\phi)](t)
        \widetilde{N}^\eta(t) dt d\phi.
    \end{eqnarray*} 
    As $\widetilde{N}^\eta(t) = e^{-\gamma t^2}$ and by using the
    Cauchy-Schwarz inequality we have
    \begin{eqnarray*}
        \abs{E[\hat{\rho}^\eta_{j,k}]- \rho_{j,k}}^2 &=&
        \abs{\frac{1}{\pi} \int_0^\pi e^{-i(j-k)\phi} \frac{1}{2\pi}
          \int_{\abs{t}> 1/\delta} \widetilde{f}_{j,k}(t)
          \mathcal{F}_1[p_\rho(\cdot|\phi)](t) dt
          d\phi}^2 \\
        &\leq& \frac{1}{\pi} \int_0^\pi \left(\frac{1}{2\pi}
            \int_{\abs{t}>1/\delta} \abs{\widetilde{f}_{j,k}(t)
              \widetilde{W}_\rho(t\cos\phi,t\sin\phi) } dt
        \right)^2d\phi.
    \end{eqnarray*}
    If $1/\delta \geq 2\sqrt{N} \geq 2s$ with $s = \sqrt{j+k+1}$,
    then whenever $t \geq 1/\delta$ we get by Lemma~\ref{lemm:estimeLaguerre}
    \begin{eqnarray*}
        |\widetilde{f}_{j,k}(t)| &= &\pi^2 |t|l_{j,k}(t/2) \\
        &\leq &\pi \abs{t} e^{-\frac 14(|t|-2s)^2}.
    \end{eqnarray*}
    On the other hand, by Propositions~\ref{prop:popetown}
    and~\ref{prop:popetown2} we have
    \begin{equation*}
        |\widetilde{W}_\rho(t\cos\phi,t\sin\phi)| \leq A(\frac{|t|}{2})
        e^{-\beta(\frac{|t|}{2})^r}
    \end{equation*}
    for $\beta \deq \frac{B}{(1+\sqrt B)^2}$ in the case $r = 2$, or for
    arbitrary $\beta < B$ and $t$ large enough in the case $0 < r < 2$. In
    both cases $A$ is a polynom of degree $4-r$. We deduce the inequality
    \begin{eqnarray*}
        \abs{E[\hat{\rho}^\eta_{j,k}] - \rho_{j,k}}^2 &\leq& C \paren{
          \int_{\frac1\delta}^\infty t^{5-r} e^{-\frac14 (t-2s)^2 - \beta
            2^{-r}t^r} dt }^2 \\
        &\leq &C (\frac1\delta)^{12-4r} e^{-\frac12 (\frac1\delta -
          2\sqrt{N})^2 - \beta 2^{1-r}(\frac1\delta)^r}
    \end{eqnarray*}
    by Lemma 8 in~\cite{Butucea&Tsybakov}, hence
    \begin{equation*}
        b_2(n)^2 \leq C N^2  (\frac1\delta)^{12-4r} e^{-\frac12 (\frac1\delta -
          2\sqrt{N})^2 - \beta 2^{1-r}(\frac1\delta)^r}
    \end{equation*}
    which covers both cases in the proposition.

\subsection{Proof of Proposition~\ref{prop:5}}    

\noindent Let us write $\sigma^2_{j,k}(n) \deq E\left|\hat{\rho}^\eta_{j,k} -
    E[\hat{\rho}^\eta_{j,k}]\right|^2$. We bound it by
\begin{eqnarray}
    \sigma^2_{j,k}(n) &=&
    E\abs{\frac{1}{n}\sum_{\ell=1}^{n} \paren{G_{j,k}
        (\frac{Y_\ell}{\sqrt\eta}, \Phi_\ell) -
        E[G_{j,k}(\frac{Y_\ell}{\sqrt\eta},
        \Phi_\ell)] } }^2 \nonumber \\
    &= &\frac{1}{n} E\abs{G_{j,k}(\frac{Y}{\sqrt\eta},\Phi) -
      E[G_{j,k}(\frac{Y}{\sqrt\eta},\Phi)]}^2 \nonumber \\
    &\leq& \frac{1}{n}
    E\abs{G_{j,k}(\frac{Y}{\sqrt\eta},\Phi)}^2. \label{eq:varbound}
\end{eqnarray}

\subparagraph{Proof of (\ref{eq:varbcpdebruit})}

\noindent For $0 < \eta \leq 1/2$, let us denote by $K_\delta$ the function with
the following Fourier transform $\widetilde{K}_\delta(t) =
\mathbb{I}(|t| \leq \frac1\delta)e^{\gamma t^2}$, then
$\widetilde{f}_{j,k}^{\eta,\delta} = \widetilde{f}_{j,k}(t)
\widetilde{K}_\delta(t)$ and we have
\begin{eqnarray*}
    \sigma^2_{j,k}(n) &\leq& \frac{1}{n}
    E\abs{f_{j,k}^{\eta,\delta}(\frac{Y}{\sqrt\eta})
      e^{-i(j-k)\Phi}}^2 \\ &\leq& \frac{1}{n}
    E\abs{f_{j,k}*K_\delta(\frac{Y}{\sqrt\eta})}^2 \\
    &\leq& \frac{1}{n} E\abs{\int f_{j,k}(t)
      K_\delta(\frac{Y}{\sqrt\eta}-t)dt}^2.
\end{eqnarray*}
By using the Cauchy-Schwarz inequality
\begin{eqnarray*}
    \sigma^2_{j,k}(n) &\leq& \frac{1}{n} \int |f_{j,k}(t)|^2 dt
    E\int\abs{K_\delta(\frac{Y}{\sqrt\eta}-t)
    }^2dt \\
    &\leq& \frac{1}{n} \int |f_{j,k}(t)|^2dt E\frac{1}{2\pi}
    \int\abs{\widetilde{
        K}_\delta(u) e^{-iu\frac{Y}{\sqrt\eta}}}^2 du \\
    &\leq& \frac{1}{n\pi} \left\|f_{j,k}\right\|^2_2
    \int_0^{1/\delta}e^{2\gamma u^2} du.
\end{eqnarray*}
Then,
\begin{equation*}
    \sigma^2(n) \leq
    \frac{C}{n\pi} \sum_{j+k=0}^{N-1} \left\|f_{j,k}\right\|^2_2
    \frac{\eta\delta}{1-\eta}e^{\frac{2\gamma}{\delta^2}}.
\end{equation*}
By Lemma~\ref{lm:4} we have $\sum_{j+k=0}^{N-1}
\left\|f_{j,k}\right\|^2_2 \leq C_2 N^{17/6}$ thus
\begin{equation*}
    \sigma^2(n) \leq \frac{C_1\eta\delta N^{17/6}}{n\pi(1-\eta)}
    e^{\frac{2\gamma}{\delta^2}}.
\end{equation*}

\subparagraph{Proof of (\ref{eq:varpeudebruit}) and
  (\ref{eq:variancesansbruit})}
%
By (\ref{estrhojk}), for $1/2 < \eta \leq 1$,
\begin{eqnarray*}
    \sigma^2_{j,k}(n) 
    &\leq& \frac{1}{n}
    E\abs{f_{j,k}^{\eta}(\frac{Y}{\sqrt\eta})e^{-i(j-k)\Phi}}^2\\
    &\leq& \frac{1}{n\pi} \int_0^\pi\int \abs{
      f_{j,k}^{\eta}(y)}^2\sqrt\eta p^\eta_\rho(\sqrt\eta
    y|\phi) dy d\phi \\
    &\leq& \frac{1}{n\pi} \left\|f_{j,k}^\eta \right\|^2_\infty
   \end{eqnarray*}
%
{For $1/2 < \eta < 1$}, by Lemma~\ref{lm:5},
%
\begin{equation*}
    \sigma^2(n) \leq \frac{C_\infty N^{1/3}}{n\pi} e^{8\gamma N}.
\end{equation*}
%
{For $\eta=1$}, by Lemma~\ref{unifbound} 
%
\begin{eqnarray*}
    \sigma^2_{j,k}(n)
    &\leq& \frac{1}{n} \int_0^\pi\int \abs{
      f_{j,k}(x)}^2 p_\rho(x,\phi)dxd\phi \\
    &\leq& \frac{C}{n} \left\|f_{j,k}\right\|^2_2
\end{eqnarray*}   
hence by Lemma~\ref{lm:4},
\begin{equation*}
    \sigma^2(n) \leq C \frac{C_2 N^{17/6}}{n}.
\end{equation*}

\subsection{Proof of Proposition~\ref{bias.varL2}}

\noindent It is easy to see that
    $$
    \mathcal{F}\left[E[\widehat W_{h}^\eta]\right](w) = \widetilde
    W_\rho(w) I(\|w\| \leq 1/h).
    $$
    \noindent
    We have, for $n$ large enough $s_n \geq z_0$ and by (\ref{eq:dec})
    \begin{eqnarray*}
        \left\| W_\rho\right\| ^2_{\overline D(s_n)} & \leq &
        C(B,r)\int _{\|z\|>s_n} \|z\|^{8-2r} \exp(-2 \beta \|z\|^r) dz \\
        & \leq &C(B,r)\int_0^{2 \pi} \int_{s_n}^\infty t^{9-2r}
        \exp(-2 \beta t^r) dt d\phi \\
        &\leq &
        C_1 s_n^{10-3r} e^{-2\beta s_n^r}, 
    \end{eqnarray*}
where $\beta <B$ and for $n$ large enough in the case $0<r<2$, respectively
$\beta=B/(1+\sqrt{B})^2$ in the case $r=2$.
    Now we write for the $\l_2$ bias of our estimator:
    \begin{eqnarray*}
        \|E[\widehat W^\eta_{h}] -W_\rho \|_{D(s_n)}^2 & \leq & \|E[\widehat
        W^\eta_{h}] -W_\rho \|_2^2 = \frac{1}{(2 \pi)^2} \| \mathcal F \left[
            E[\widehat W^\eta_{h}]\right] - \widetilde W_\rho\|_2^2\\
        &=& \frac{1}{(2 \pi)^2} \int \abs{ \widetilde W_\rho (w) }^2
        I(\|w\|> 1/h) \, dw \\
        & \leq&  \frac{C^{2 }(B,r)}{(2 \pi)^2} \int_{\|w\|> 1/h}
        \|w\|^{2(4-r)} e^{-2^{1-r} \beta \|w\|^r}\,dw \\
        &\leq & C_{2 } h^{3r-10} e^{- \frac{2^{1-r} \beta}{ h^r}},
    \end{eqnarray*}
    by the assumption on our class and (\ref{eq:reg}), for
    $0<r<2$. The case $r=2$ is similar.

    \noindent
    As for the variance of our estimator:
    \begin{eqnarray}
       V\left[\widehat W_{h}^\eta \right]
        &=& E\left[\left\|\widehat W_{h}^\eta - E \left[\widehat
                    W_{h}^\eta \right]\right\|_{D(s_n)}^2
        \right]
        \nonumber  \\
        &=&
        \frac{1}{\pi^2 n} \left\{E\left[\left\|K_h^\eta \left([\cdot,
                        \Phi]-\frac{Y}{\sqrt{\eta}}
                    \right) \right\|_{D(s_n)}^2 \right]
            \right.\nonumber\\
           &&-\left. \left\|E\left[K_h^n\left([\cdot, \Phi]-\frac{Y}{\sqrt{\eta}}
                    \right) \right]\right\|_{D(s_n)}^2 \right\}.
        \label{varL2}
    \end{eqnarray}
    On the one hand, by using two-dimensional Plancherel formula and the
    Fourier transform shown above, we get:
    \begin{equation}
        \left\|E\left[K_h^n\left([\cdot, \Phi]-\frac{Y}{\sqrt{\eta}}
                \right) \right]\right\|_{D(s_n)}^2 \leq \pi^2\int
        |W_\rho(w)|^2 dw \leq \pi^2.
        \label{varL21}
    \end{equation}
    In the last inequality we have used the fact that $\| W_\rho
    \|_2^2 = \mathrm{Tr}(\rho^2)\leq 1$ where $\rho$ is the density
    matrix corresponding to the Wigner function $W_\rho$.
    \noindent
    On the other hand, the dominant term in the variance will be given by
    \begin{eqnarray*}
       && E\left[\left\|K_h^\eta \left([\cdot, \Phi]-\frac{Y}{\sqrt{\eta}}
                \right) \right\|_{D(s_n)}^2 \right]\\
        &= &
        \int_0^\pi\int \int_{D(s_n)} \left(K_h^\eta([z,\phi]-y/\sqrt{\eta})
        \right)^2 dz p_\rho^\eta(y,\phi) dy d\phi\\
        & = & \int_0^\pi\int_{D(s_n)} \int \left(K_h^\eta(u)\right)^2 \sqrt{\eta}
        p_\rho^\eta(([z,\phi]-u)\sqrt{\eta},\phi) du dz d\phi\\
        & = &
        \int \left(K_h^\eta(u)\right)^2 \int_{D(s_n)} \int_0^\pi p_\rho(\cdot,\phi)
        \ast NN^\eta ([z,\phi]-u)d \phi dz du\\
        & \leq &
        M(\eta) \pi s_n^2 \int (K_h^\eta (u))^2 du,
    \end{eqnarray*}
    using Lemma~\ref{unifbound} below and the constant $M(\eta)>0$
    depending only on $\eta$, defined therein.  Indeed, let us note
    that $\sqrt{\eta} p_\rho^\eta(\cdot \sqrt{\eta},\phi)$ is the
    density of $Y/\sqrt{\eta} = X+\sqrt{(1-\eta)/(2\eta)} \varepsilon$
    and let us call $NN^{\eta}$ the Gaussian density of the noise as
    normalized in this last equation.

    \noindent
    Let us first compute, by Plancherel formula, $\|K_h^\eta\|_2^2$
    and get
    \begin{eqnarray*}
        \|K^\eta _h\|_2^2 &=& \frac{1}{2 \pi} \int |\widetilde K^\eta_h(t)|^2 dt
        =\frac{1}{2 \pi} \int_{|t| \leq 1/h} \frac{t^2}{4\widetilde
          N^2(t\sqrt{(1-\eta)/(2\eta)})} dt\\
        &=& \frac{1}{4 \pi}  \int_{0}^{1/h} t^2 \exp\left(t^2
            \frac{1-\eta}{2\eta} \right) dt\\
        & = & \frac{1}{4 \pi h} \frac{\eta}{1 -\eta} \exp
        \left(\frac{1-\eta}{2\eta h^2} \right)
        (1+o(1)), {\rm \ as \ } h\to 0.
    \end{eqnarray*}
    We replace in the second order moment, then as $h \to 0$
    \begin{eqnarray}
        E\left[\left\|K_h^\eta \left([\cdot,
                    \Phi]-\frac{Y}{\sqrt{\eta}}\right)
            \right\|_{D(s_n)}^2 \right]
        & \leq &
        \frac{M(\eta) s_n^2}{16 \gamma h}
        \exp\left(\frac{2 \gamma}{h^2} \right)(1+o(1)).
        \label{varL22}
    \end{eqnarray}
    The result about the variance of the estimator is obtained from
    (\ref{varL2})-(\ref{varL22}).

\begin{lemma}
    \label{unifbound}
    For every $\rho \in \mathcal{R}(B,r)$ and $0< \eta < 1$, we have
    that the corresponding probability density $p_{\rho}$ satisfies
    \begin{eqnarray*}
        0& \leq& \int_0^\pi p_\rho (\cdot, \phi) \ast NN^\eta (x) d \phi
        \leq M(\eta),\\
        0 &\leq &\int_0^\pi p_\rho (x, \phi) d \phi\leq C
    \end{eqnarray*}
    for all $x \in \mathbb{R}$ eventually depending on $\phi$, where
    $M(\eta)>0$ is a constant depending only on fixed $\eta$ and $C>0$.
\end{lemma}

\begin{proof}
    Indeed, using inverse Fourier transform and the fact that
    $\abs{\widetilde{W}_{\rho}(w)}\leq 1$ we get:
    \begin{eqnarray*}
       && \abs{\int_0^\pi p_\rho (\cdot, \phi) \ast NN^\eta (x) d \phi}\\
        & \leq & \abs{\int_0^\pi\frac{1}{2 \pi}\int e^{-it x}
            \mathcal{F}_1[p_\rho (\cdot,\phi)](t) \cdot \widetilde{NN}^\eta (t)dt
            d \phi}\\
        & \leq & c(\eta) \int_0^\pi \int \abs{\widetilde W _\rho (t
            \cos\phi, t \sin\phi)}
        \exp \left( - \frac{t^2 (1 -\eta)}{4 \eta}\right) dt d\phi \\
        & \leq & c(\eta) \int \frac{1}{\|w\|} \abs{\widetilde W _\rho
            (w)}
        \exp \left(-\frac{\|w\|^2(1-\eta)}{4 \eta} \right) dw \leq M(\eta),
    \end{eqnarray*}
    where $c(\eta), \, M(\eta)$ are positive constants depending only
    on $\eta \in (0,1)$. \hfill
\end{proof}
\bibliographystyle{plain}
\bibliography{Wignerbib}

\begin{thebibliography}{10}

\bibitem{Artiles&Gill&Guta}
L.~Artiles, R.~Gill, and M.~Gu{\c{t}}{\u{a}}.
\newblock An invitation to quantum tomography.
\newblock {\em J. Royal Statist. Soc. B (Methodological)}, 67:109--134, 2005.

\bibitem{Aubry:2007fk}
J.~M. Aubry.
\newblock Ultrarapidly decreasing ultradifferentiable functions, {Wigner}
  distributions and density matrices.
\newblock Submitted to J. London Math. Soc., 2008.

\bibitem{BDPS}
K.~Banaszek, G.~M. D'Ariano, M.~G.~A. Paris, and M.~F. Sacchi.
\newblock Maximum-likelihood estimation of the density matrix.
\newblock {\em Physical Review A}, 61(R10304), 2000.

\bibitem{Butucea&Guta&Artiles}
C.~Butucea, M.~Gu{\c{t}}{\u{a}}, and L.~Artiles.
\newblock Minimax and adaptive estimation of the {Wigner} function in quantum
  homodyne tomography with noisy data.
\newblock {\em Ann. Statist.}, 35(2):465--494, 2007.

\bibitem{Butucea&Tsybakov}
C.~Butucea and A.~B. Tsybakov.
\newblock Sharp optimality for density deconvolution with dominating bias. i
  and ii.
\newblock {\em Theory Probab. Appl.}, 2007.

\bibitem{Cavalier}
L.~Cavalier.
\newblock Efficient estimation of a density in a problem of tomography.
\newblock {\em Ann. Statist.}, 28:630--647, 2000.

\bibitem{DAriano.1}
G.~M. D'Ariano.
\newblock Tomographic measurement of the density matrix of the radiation field.
\newblock {\em Quantum Semiclass. Optics}, 7:693--704, 1995.

\bibitem{DAriano.5}
G.~M. D'Ariano.
\newblock Tomographic methods for universal estimation in quantum optics.
\newblock In {\em International School of Physics Enrico Fermi}, volume 148.
  IOS Press, 2002.

\bibitem{DAriano.2}
G.~M. D'Ariano, U.~Leonhardt, and H.~Paul.
\newblock Homodyne detection of the density matrix of the radiation field.
\newblock {\em Phys. Rev. A}, 52:R1801--R1804, 1995.

\bibitem{DAriano.0}
G.~M. D'Ariano, C.~Macchiavello, and M.~G.~A. Paris.
\newblock Detection of the density matrix through optical homodyne tomography
  without filtered back projection.
\newblock {\em Phys. Rev. A}, 50:4298--4302, 1994.

\bibitem{DP}
G.~M. D'Ariano and M.~G.~A. Paris.
\newblock Adaptive quantum tomography.
\newblock {\em Physical Review A}, 60(518), 1999.

\bibitem{DMS}
G.M. D'Ariano, L.~Maccone, and M.~F. Sacchi.
\newblock Homodyne tomography and the reconstruction of quantum states of
  light, 2005.

\bibitem{Goldenshluger&Spokoiny}
A.~Goldenshluger and V.~Spokoiny.
\newblock On the shape-from-moments problem and recovering edges from noisy
  {R}adon data.
\newblock {\em Probab. Theory Related Fields}, 128(1):123--140, 2004.

\bibitem{Guta&Artiles}
M.~Gu{\c{t}}{\u{a}} and L.~Artiles.
\newblock Minimax estimation of the {Wigner} in quantum homodyne tomography
  with ideal detectors.
\newblock {\em Math. Methods Statist.}, 16(1):1--15, 2007.

\bibitem{Guta}
M.~I. Gu{\c{t}}{\u{a}}.
\newblock Maximum likelihood estimation of the density matrix through quantum
  tomography.
\newblock Manuscript, 2007.

\bibitem{Korostelev&Tsybakov}
A.~P. Korostel{\"e}v and A.~B. Tsybakov.
\newblock {\em Minimax theory of image reconstruction}, volume~82 of {\em
  Lecture Notes in Statistics}.
\newblock Springer-Verlag, New York, 1993.

\bibitem{Krasikov:2005db}
I.~Krasikov.
\newblock Inequalities for {L}aguerre polynomials.
\newblock {\em East J. Approx.}, 11(3):257--268, 2005.

\bibitem{Krasikov:2007uq}
I.~Krasikov.
\newblock Inequalities for orthonormal {Laguerre} polynomials.
\newblock {\em J. Approx. Theory}, 144(1):1--26, 2007.

\bibitem{Leonhardt}
U.~Leonhardt.
\newblock {\em Measuring the Quantum State of Light}.
\newblock Cambridge University Press, 1997.

\bibitem{DAriano.3}
U.~Leonhardt, H.~Paul, and G.~M. D'Ariano.
\newblock Tomographic reconstruction of the density matrix via pattern
  functions.
\newblock {\em Phys. Rev. A}, 52:4899--4907, 1995.

\bibitem{Johnstone&Silverman}
Johnstone~Iain M. and Silverman~Bernard W.
\newblock Speed of estimation in positron emission tomography and related
  inverse problems.
\newblock {\em Ann. Statist.}, 18(1):251--280, 1990.

\bibitem{Meziani}
K.~Meziani.
\newblock Nonparametric estimation of the purity of a quantum state in quantum
  homodyne tomography with noisy data.
\newblock {\em Mathematical Methods of Statistics}, 16(4):1--15, 2007.

\bibitem{Richter}
T.~Richter.
\newblock Pattern functions used in tomographic reconstruction of photon
  statistics revisited.
\newblock {\em Phys. Lett. A}, 211:327--330, 1996.

\bibitem{Richter1}
T.~Richter.
\newblock Realistic pattern functions for optical homodyne tomography and
  determination of specific expectation values.
\newblock {\em Physical Review A}, 61(063819), 2000.

\bibitem{Smitheybis}
D.~J. Smithey, M.~Beck, M.~J. Cooper, and M.~G. Raymer.
\newblock Experimental determination of number-phase uncertainty relations.
\newblock {\em Optics Letters}, 18:1259--1261, 1993.

\bibitem{Smithey}
D.~T. Smithey, M.~Beck, M.~G. Raymer, and A.~Faridani.
\newblock Measurement of the {Wigner} distribution and the density matrix of a
  light mode using optical homodyne tomography: Application to squeezed states
  and the vacuum.
\newblock {\em Phys. Rev. Lett.}, 70:1244--1247, 1993.

\bibitem{Szego:1959uq}
G.~Szeg{\"o}.
\newblock {\em Orthogonal polynomials}.
\newblock American Mathematical Society Colloquium Publications, Vol. 23.
  Revised ed. American Mathematical Society, Providence, R.I., 1959.

\bibitem{Vardi}
Y.~Vardi, L.~A. Shepp, and L.~Kaufman.
\newblock A statistical model for positron emission tomography.
\newblock {\em J. Am. Stat. Assoc.}, 80:8--37, 1985.

\bibitem{Vogel&Risken}
K.~Vogel and H.~Risken.
\newblock Determination of quasiprobability distributions in terms of
  probability distributions for the rotated quadrature phase.
\newblock {\em Phys. Rev. A}, 40:2847--2849, 1989.

\end{thebibliography}
\end{document}